\documentclass[reqno]{amsart}
\usepackage{amssymb}
\usepackage{eucal}
\usepackage{amsmath}
\usepackage{amscd}
\usepackage[dvips]{color}
\usepackage{multicol}
\usepackage[all]{xy}           
\usepackage{graphicx}
\usepackage{color}
\usepackage{colordvi}
\usepackage{xspace}
\usepackage{bookmark}

\usepackage[active]{srcltx} 

\usepackage{hyperref}
\usepackage{lipsum}

\newcommand{\nc}{\newcommand}
\newcommand{\delete}[1]{}
\nc{\mfootnote}[1]{\footnote{#1}} 
\nc{\todo}[1]{\tred{To do:} #1}

\delete{
\nc{\mlabel}[1]{\label{#1}}  
\nc{\mcite}[1]{\cite{#1}}  
\nc{\mref}[1]{\ref{#1}}  
\nc{\mbibitem}[1]{\bibitem{#1}} 
}

\nc{\mlabel}[1]{\label{#1}  
{\hfill \hspace{1cm}{\bf{{\ }\hfill(#1)}}}}
\nc{\mcite}[1]{\cite{#1}{{\bf{{\ }(#1)}}}}  
\nc{\mref}[1]{\ref{#1}{{\bf{{\ }(#1)}}}}  
\nc{\mbibitem}[1]{\bibitem[\bf #1]{#1}} 

\newtheorem{theorem}{Theorem}[section]
\newtheorem{definition}[theorem]{Definition}
\newtheorem{lemma}[theorem]{Lemma}

\newtheorem{prop-def}[theorem]{Proposition-Definition}

\newtheorem{remark}[theorem]{Remark}

\newtheorem{condition}[theorem]{Assumption}

\nc{\tred}[1]{\textcolor{red}{#1}}
\nc{\tblue}[1]{\textcolor{blue}{#1}}
\nc{\tgreen}[1]{\textcolor{green}{#1}}
\nc{\tpurple}[1]{\textcolor{purple}{#1}}
\nc{\btred}[1]{\textcolor{red}{\bf #1}}
\nc{\btblue}[1]{\textcolor{blue}{\bf #1}}
\nc{\btgreen}[1]{\textcolor{green}{\bf #1}}
\nc{\btpurple}[1]{\textcolor{purple}{\bf #1}}

\nc{\li}[1]{\textcolor{red}{Xiaomin:#1}}
\nc{\cm}[1]{\textcolor{blue}{Chengming: #1}}

\nc{\twovec}[2]{\left(\begin{array}{c} #1 \\ #2\end{array} \right )}
\nc{\threevec}[3]{\left(\begin{array}{c} #1 \\ #2 \\ #3 \end{array}\right )}
\nc{\twomatrix}[4]{\left(\begin{array}{cc} #1 & #2\\ #3 & #4 \end{array} \right)}
\nc{\threematrix}[9]{{\left(\begin{matrix} #1 & #2 & #3\\ #4 & #5 & #6 \\ #7 & #8 & #9 \end{matrix} \right)}}
\nc{\twodet}[4]{\left|\begin{array}{cc} #1 & #2\\ #3 & #4 \end{array} \right|}

\nc{\rk}{\mathrm{r}}

\nc{\gensp}{V} 
\nc{\relsp}{\Lambda} 
\nc{\leafsp}{X}    
\nc{\treesp}{\overline{\calt}} 

\nc{\vin}{{\mathrm Vin}}    
\nc{\lin}{{\mathrm Lin}}    

\nc{\gop}{{\,\omega\,}}     
\nc{\gopb}{{\,\nu\,}}
\nc{\svec}[2]{{\tiny\left(\begin{matrix}#1\\
#2\end{matrix}\right)\,}}  
\nc{\ssvec}[2]{{\tiny\left(\begin{matrix}#1\\
#2\end{matrix}\right)\,}} 

\nc{\su}{\mathrm{Su}}
\nc{\tsu}{\mathrm{TSu}}
\nc{\TSu}{\mathrm{TSu}}
\nc{\eval}[1]{{#1}_{\big|D}}
\nc{\oto}{\leftrightarrow}

\nc{\oaset}{\mathbf{O}^{\rm alg}}
\nc{\omset}{\mathbf{O}^{\rm mod}}
\nc{\oamap}{\Phi^{\rm alg}}
\nc{\ommap}{\Phi^{\rm mod}}
\nc{\ioaset}{\mathbf{IO}^{\rm alg}}
\nc{\iomset}{\mathbf{IO}^{\rm mod}}
\nc{\ioamap}{\Psi^{\rm alg}}
\nc{\iommap}{\Psi^{\rm mod}}

\nc{\suc}{{successor}\xspace} \nc{\Suc}{{Successor}\xspace}
\nc{\sucs}{{successors}\xspace} \nc{\Sucs}{{Successors}\xspace}
\nc{\Tsuc}{{T-successor}\xspace}
\nc{\Tsucs}{{T-successors}\xspace} \nc{\Lsuc}{{L-successor}\xspace}
\nc{\Lsucs}{{L-successors}\xspace} \nc{\Rsuc}{{R-successor}\xspace}
\nc{\Rsucs}{{R-successors}\xspace}

\nc{\bia}{{$\mathcal{P}$-bimodule ${\bf k}$-algebra}\xspace}
\nc{\bias}{{$\mathcal{P}$-bimodule ${\bf k}$-algebras}\xspace}

\nc{\rmi}{{\mathrm{I}}}
\nc{\rmii}{{\mathrm{II}}}
\nc{\rmiii}{{\mathrm{III}}}

\nc{\pll}{\beta}
\nc{\plc}{\epsilon}

\nc{\ass}{{\mathit{Ass}}}
\nc{\lie}{{\mathit{Lie}}}
\nc{\comm}{{\mathit{Comm}}}
\nc{\dend}{{\mathit{Dend}}}
\nc{\zinb}{{\mathit{Zinb}}}
\nc{\tdend}{{\mathit{TDend}}}
\nc{\prelie}{{\mathit{preLie}}}
\nc{\postlie}{{\mathit{PostLie}}}
\nc{\quado}{{\mathit{Quad}}}
\nc{\octo}{{\mathit{Octo}}}
\nc{\ldend}{{\mathit{ldend}}}
\nc{\lquad}{{\mathit{LQuad}}}

 \nc{\adec}{\check{;}} \nc{\aop}{\alpha}
\nc{\dftimes}{\widetilde{\otimes}} \nc{\dfl}{\succ} \nc{\dfr}{\prec}
\nc{\dfc}{\circ} \nc{\dfb}{\bullet} \nc{\dft}{\star}
\nc{\dfcf}{{\mathbf k}} \nc{\apr}{\ast} \nc{\spr}{\cdot}
\nc{\twopr}{\circ} \nc{\tspr}{\star} \nc{\sempr}{\ast}
\nc{\disp}[1]{\displaystyle{#1}}
\nc{\bin}[2]{ (_{\stackrel{\scs{#1}}{\scs{#2}}})}  
\nc{\binc}[2]{ \left (\!\! \begin{array}{c} \scs{#1}\\
    \scs{#2} \end{array}\!\! \right )}  
\nc{\bincc}[2]{  \left ( {\scs{#1} \atop
    \vspace{-.5cm}\scs{#2}} \right )}  
\nc{\sarray}[2]{\begin{array}{c}#1 \vspace{.1cm}\\ \hline
    \vspace{-.35cm} \\ #2 \end{array}}
\nc{\bs}{\bar{S}} \nc{\dcup}{\stackrel{\bullet}{\cup}}
\nc{\dbigcup}{\stackrel{\bullet}{\bigcup}} \nc{\etree}{\big |}
\nc{\la}{\longrightarrow} \nc{\fe}{\'{e}} \nc{\rar}{\rightarrow}
\nc{\dar}{\downarrow} \nc{\dap}[1]{\downarrow
\rlap{$\scriptstyle{#1}$}} \nc{\uap}[1]{\uparrow
\rlap{$\scriptstyle{#1}$}} \nc{\defeq}{\stackrel{\rm def}{=}}
\nc{\dis}[1]{\displaystyle{#1}} \nc{\dotcup}{\,
\displaystyle{\bigcup^\bullet}\ } \nc{\sdotcup}{\tiny{
\displaystyle{\bigcup^\bullet}\ }} \nc{\hcm}{\ \hat{,}\ }
\nc{\hcirc}{\hat{\circ}} \nc{\hts}{\hat{\shpr}}
\nc{\lts}{\stackrel{\leftarrow}{\shpr}}
\nc{\rts}{\stackrel{\rightarrow}{\shpr}} \nc{\lleft}{[}
\nc{\lright}{]} \nc{\uni}[1]{\tilde{#1}} \nc{\wor}[1]{\check{#1}}
\nc{\free}[1]{\bar{#1}} \nc{\den}[1]{\check{#1}} \nc{\lrpa}{\wr}
\nc{\curlyl}{\left \{ \begin{array}{c} {} \\ {} \end{array}
    \right .  \!\!\!\!\!\!\!}
\nc{\curlyr}{ \!\!\!\!\!\!\!
    \left . \begin{array}{c} {} \\ {} \end{array}
    \right \} }
\nc{\leaf}{\ell}       
\nc{\longmid}{\left | \begin{array}{c} {} \\ {} \end{array}
    \right . \!\!\!\!\!\!\!}
\nc{\ot}{\otimes} \nc{\sot}{{\scriptstyle{\ot}}}
\nc{\otm}{\overline{\ot}}
\nc{\ora}[1]{\stackrel{#1}{\rar}}
\nc{\ola}[1]{\stackrel{#1}{\la}}
\nc{\pltree}{\calt^\pl}
\nc{\epltree}{\calt^{\pl,\NC}}
\nc{\rbpltree}{\calt^r}
\nc{\scs}[1]{\scriptstyle{#1}} \nc{\mrm}[1]{{\rm #1}}
\nc{\dirlim}{\displaystyle{\lim_{\longrightarrow}}\,}
\nc{\invlim}{\displaystyle{\lim_{\longleftarrow}}\,}
\nc{\mvp}{\vspace{0.5cm}} \nc{\svp}{\vspace{2cm}}
\nc{\vp}{\vspace{8cm}} \nc{\proofbegin}{\noindent{\bf Proof: }}
\nc{\proofend}{$\blacksquare$ \vspace{0.5cm}}
\nc{\freerbpl}{{F^{\mathrm RBPL}}}
\nc{\sha}{{\mbox{\cyr X}}}  
\nc{\ncsha}{{\mbox{\cyr X}^{\mathrm NC}}} \nc{\ncshao}{{\mbox{\cyr
X}^{\mathrm NC,\,0}}}
\nc{\shpr}{\diamond}    
\nc{\shprm}{\overline{\diamond}}    
\nc{\shpro}{\diamond^0}    
\nc{\shprr}{\diamond^r}     
\nc{\shpra}{\overline{\diamond}^r}
\nc{\shpru}{\check{\diamond}} \nc{\catpr}{\diamond_l}
\nc{\rcatpr}{\diamond_r} \nc{\lapr}{\diamond_a}
\nc{\sqcupm}{\ot}
\nc{\lepr}{\diamond_e} \nc{\vep}{\varepsilon} \nc{\labs}{\mid\!}
\nc{\rabs}{\!\mid} \nc{\hsha}{\widehat{\sha}}
\nc{\lsha}{\stackrel{\leftarrow}{\sha}}
\nc{\rsha}{\stackrel{\rightarrow}{\sha}} \nc{\lc}{\lfloor}
\nc{\rc}{\rfloor}
\nc{\tpr}{\sqcup}
\nc{\nctpr}{\vee}
\nc{\plpr}{\star}
\nc{\rbplpr}{\bar{\plpr}}
\nc{\sqmon}[1]{\langle #1\rangle}
\nc{\forest}{\calf}
\nc{\altx}{\Lambda_X} \nc{\vecT}{\vec{T}} \nc{\onetree}{\bullet}
\nc{\Ao}{\check{A}}
\nc{\seta}{\underline{\Ao}}
\nc{\deltaa}{\overline{\delta}}
\nc{\trho}{\tilde{\rho}}

\nc{\rpr}{\circ}
\nc{\dpr}{{\tiny\diamond}}
\nc{\rprpm}{{\rpr}}

\nc{\mmbox}[1]{\mbox{\ #1\ }} \nc{\ann}{\mrm{ann}}
\nc{\Aut}{\mrm{Aut}} \nc{\can}{\mrm{can}}
\nc{\twoalg}{{two-sided algebra}\xspace}
\nc{\colim}{\mrm{colim}}
\nc{\Cont}{\mrm{Cont}} \nc{\rchar}{\mrm{char}}
\nc{\cok}{\mrm{coker}} \nc{\dtf}{{R-{\rm tf}}} \nc{\dtor}{{R-{\rm
tor}}}

\nc{\depth}{{\mrm d}}
\nc{\Div}{{\mrm Div}} \nc{\End}{\mrm{End}} \nc{\Ext}{\mrm{Ext}}
\nc{\Fil}{\mrm{Fil}} \nc{\Frob}{\mrm{Frob}} \nc{\Gal}{\mrm{Gal}}
\nc{\GL}{\mrm{GL}} \nc{\Hom}{\mrm{Hom}} \nc{\hsr}{\mrm{H}}
\nc{\hpol}{\mrm{HP}} \nc{\id}{\mrm{id}} \nc{\im}{\mrm{im}}
\nc{\incl}{\mrm{incl}} \nc{\length}{\mrm{length}}
\nc{\LR}{\mrm{LR}} \nc{\mchar}{\rm char} \nc{\NC}{\mrm{NC}}
\nc{\mpart}{\mrm{part}} \nc{\pl}{\mrm{PL}}
\nc{\ql}{{\QQ_\ell}} \nc{\qp}{{\QQ_p}}
\nc{\rank}{\mrm{rank}} \nc{\rba}{\rm{RBA }} \nc{\rbas}{\rm{RBAs }}
\nc{\rbpl}{\mrm{RBPL}}
\nc{\rbw}{\rm{RBW }} \nc{\rbws}{\rm{RBWs }} \nc{\rcot}{\mrm{cot}}
\nc{\rest}{\rm{controlled}\xspace}
\nc{\rdef}{\mrm{def}} \nc{\rdiv}{{\rm div}} \nc{\rtf}{{\rm tf}}
\nc{\rtor}{{\rm tor}} \nc{\res}{\mrm{res}} \nc{\SL}{\mrm{SL}}
\nc{\Spec}{\mrm{Spec}} \nc{\tor}{\mrm{tor}} \nc{\Tr}{\mrm{Tr}}
\nc{\mtr}{\mrm{sk}}

\nc{\ab}{\mathbf{Ab}} \nc{\Alg}{\mathbf{Alg}}
\nc{\Algo}{\mathbf{Alg}^0} \nc{\Bax}{\mathbf{Bax}}
\nc{\Baxo}{\mathbf{Bax}^0} \nc{\RB}{\mathbf{RB}}
\nc{\RBo}{\mathbf{RB}^0} \nc{\BRB}{\mathbf{RB}}
\nc{\Dend}{\mathbf{DD}} \nc{\bfk}{{\bf k}} \nc{\bfone}{{\bf 1}}
\nc{\base}[1]{{a_{#1}}} \nc{\detail}{\marginpar{\bf More detail}
    \noindent{\bf Need more detail!}
    \svp}
\nc{\Diff}{\mathbf{Diff}} \nc{\gap}{\marginpar{\bf
Incomplete}\noindent{\bf Incomplete!!}
    \svp}
\nc{\FMod}{\mathbf{FMod}} \nc{\mset}{\mathbf{MSet}}
\nc{\rb}{\mathrm{RB}} \nc{\Int}{\mathbf{Int}}
\nc{\Mon}{\mathbf{Mon}}
\nc{\remarks}{\noindent{\bf Remarks: }}
\nc{\OS}{\mathbf{OS}} 
\nc{\Rep}{\mathbf{Rep}}
\nc{\Rings}{\mathbf{Rings}} \nc{\Sets}{\mathbf{Sets}}
\nc{\DT}{\mathbf{DT}}

\nc{\BA}{{\mathbb A}} \nc{\CC}{{\mathbb C}} \nc{\DD}{{\mathbb D}}
\nc{\EE}{{\mathbb E}} \nc{\FF}{{\mathbb F}} \nc{\GG}{{\mathbb G}}
\nc{\HH}{{\mathbb H}} \nc{\LL}{{\mathbb L}} \nc{\NN}{{\mathbb N}}
\nc{\QQ}{{\mathbb Q}} \nc{\RR}{{\mathbb R}} \nc{\BS}{{\mathbb{S}}} \nc{\TT}{{\mathbb T}}
\nc{\VV}{{\mathbb V}} \nc{\ZZ}{{\mathbb Z}}

\nc{\calao}{{\mathcal A}} \nc{\cala}{{\mathcal A}}
\nc{\calc}{{\mathcal C}} \nc{\cald}{{\mathcal D}}
\nc{\cale}{{\mathcal E}} \nc{\calf}{{\mathcal F}}
\nc{\calfr}{{{\mathcal F}^{\,r}}} \nc{\calfo}{{\mathcal F}^0}
\nc{\calfro}{{\mathcal F}^{\,r,0}} \nc{\oF}{\overline{F}}
\nc{\calg}{{\mathcal G}} \nc{\calh}{{\mathcal H}}
\nc{\cali}{{\mathcal I}} \nc{\calj}{{\mathcal J}}
\nc{\call}{{\mathcal L}} \nc{\calm}{{\mathcal M}}
\nc{\caln}{{\mathcal N}} \nc{\calo}{{\mathcal O}}
\nc{\calp}{{\mathcal P}} \nc{\calq}{{\mathcal Q}} \nc{\calr}{{\mathcal R}}
\nc{\calt}{{\mathcal T}} \nc{\caltr}{{\mathcal T}^{\,r}}
\nc{\calu}{{\mathcal U}} \nc{\calv}{{\mathcal V}}
\nc{\calw}{{\mathcal W}} \nc{\calx}{{\mathcal X}}
\nc{\CA}{\mathcal{A}}

\nc{\fraka}{{\mathfrak a}} \nc{\frakB}{{\mathfrak B}}
\nc{\frakb}{{\mathfrak b}} \nc{\frakd}{{\mathfrak d}}
\nc{\oD}{\overline{D}}
\nc{\frakF}{{\mathfrak F}} \nc{\frakg}{{\mathfrak g}}
\nc{\frakm}{{\mathfrak m}} \nc{\frakM}{{\mathfrak M}}
\nc{\frakMo}{{\mathfrak M}^0} \nc{\frakp}{{\mathfrak p}}
\nc{\frakS}{{\mathfrak S}} \nc{\frakSo}{{\mathfrak S}^0}
\nc{\fraks}{{\mathfrak s}} \nc{\os}{\overline{\fraks}}
\nc{\frakT}{{\mathfrak T}}
\nc{\oT}{\overline{T}}
\nc{\frakX}{{\mathfrak X}} \nc{\frakXo}{{\mathfrak X}^0}
\nc{\frakx}{{\mathbf x}}
\nc{\frakTx}{\frakT}      
\nc{\frakTa}{\frakT^a}        
\nc{\frakTxo}{\frakTx^0}   
\nc{\caltao}{\calt^{a,0}}   
\nc{\ox}{\overline{\frakx}} \nc{\fraky}{{\mathfrak y}}
\nc{\frakz}{{\mathfrak z}} \nc{\oX}{\overline{X}}

\font\cyr=wncyr10

\nc{\redtext}[1]{\textcolor{red}{#1}}

\makeatletter
\g@addto@macro{\endabstract}{\@setabstract}
\newcommand{\authorfootnotes}{\renewcommand\thefootnote{\@fnsymbol\c@footnote}}%
\makeatother
\begin{document}
\begin{center}
  \LARGE
\textbf{Biderivations of finite dimensional complex simple Lie algebras }

  \normalsize
  \authorfootnotes
Xiaomin Tang   \footnote{Corresponding author: {\it X. Tang. Email:} x.m.tang@163.com}
\par \bigskip

   \textsuperscript{1} Department of Mathematics, Heilongjiang University,
Harbin, 150080, P. R. China   \par

\end{center}


\begin{abstract}

In this paper, we prove that a biderivation of a finite dimensional complex simple Lie algebra without the restriction of skewsymmetric is inner. As an application, the biderivation of a general linear Lie algebra is presented. In particular, we find a class of a non-inner and non-skewsymmetric biderivations. Furthermore, we also get the forms of linear commuting maps on the finite dimensional complex simple Lie algebra or general linear Lie algebra.

\vspace{2mm}

\noindent{\it Keywords:} biderivation, simple Lie algebra, general linear Lie algebra, linear commuting map

\noindent{\it AMS subject classifications:} 17B05, 17B40, 17B65.

\end{abstract}

\setcounter{section}{0}
{\ }

 \baselineskip=20pt

\section{Introduction and  preliminary results}

Derivations and generalized derivations are very important subjects in the
study of both algebras and their generalizations. In recent years, biderivations have aroused the interest of a great many authors, see  \cite{Bre1995,Ben2009,Chen2016,Du2013,Gho2013,Hanw,WD1,WD3,WD2,Xu2015}. In \cite{Bre1995}, Bre\v{s}ar et. al. showed that all biderivations on
commutative prime rings are inner biderivations, and determined the biderivations of semiprime
rings. The notation of biderivations of Lie algebras was introduced in \cite{WD3}. In addition, Wang et. al. proved that the skewsymmetric biderivations of the Schr\"{o}dinger-Virasoro algebra are inner biderivations \cite{WD1}. Furthermore, Xia et. al. showed that the skewsymmetric super-biderivations on the super-Virasoro Algebras are inner  \cite{WD2}. In \cite{Chen2016}, Chen pointed out every skewsymmetric biderivation of a simple generalized Witt algebra over a field of characteristic zero is also inner. In \cite{Hanw}, Han et. al. determined all the skewsymmetric biderivations of $W(a,b)$ and found that there exist non-inner biderivations. It may be useful and interesting for computing the biderivations of some important Lie (super-)algebras. Of course, we should first consider the finite dimensional complex simple Lie algebra. In fact,  Wang et. al. showed that the skewsymmetric biderivation of a finite dimensional complex simple Lie algebra is inner \cite{WD2}, but the problem about the general biderivation  (without the restriction of skewsymmetric) is still open. We will study this problem in this paper.

For an arbitrary Lie algebra $L$, we recall that a bilinear map $f : L\times L \rightarrow L$ is a
biderivation of $L$ if it is a derivation with respect to both components. More precisely, one has
\begin{definition}
Suppose that $L$ is a Lie algebra. A bilinear map $f: L\times L\rightarrow L$ is called a biderivation if it satisfies
\begin{eqnarray}
f([x,y],z)=[x,f(y,z)]+[f(x,z),y], \label{2der}\\
f(x,[y,z])=[f(x,y),z]+[y,f(x,z)] \label{1der}
\end{eqnarray}
for all $x, y, z\in L$.
\end{definition}
Let $\lambda\in \mathbb{C}$, the bilinear map $f: L\times L\rightarrow L$ sending $(x,y)$ to $\lambda [x,y]$, is a biderivation of $L$. It will be called an inner biderivation of $L$.
It is well known that every derivation of a finite dimensional complex  simple Lie algebra is inner. There is a natural question: what is the biderivation of a finite dimensional complex simple Lie algebra?  It seems that this question has already answered by \cite{WD3}, i.e., every biderivation of a finite dimensional complex simple Lie algebra is inner. The following assumption is a pivotal Lemma in \cite{WD3} (see Lemma 2.2 of  \cite{WD3}).

\begin{condition}\cite{WD3}\label{zhangd2011}
Let $f$ be a biderivation of a Lie algebra $L$. Then one has
\begin{equation}\label{fourel}
[f(x,y),[u,v]]=[[x,y],f(u,v)], \ \ \forall x,y,u,v\in L.
\end{equation}
\end{condition}

Unfortunately, we see that Assumption  \ref{zhangd2011} is invalid by a simple computation. In fact,
since $f$ is a derivation in the first argument, it can be seen that
$$
f([x,u],[y,v])=[f(x,[y,v]),u]+[x,f(u,[y,v])].
$$
Using the fact that $f$ is a derivation in the second argument, one has
$$
f([x,u],[y,v])=[[f(x,y),v],u]+[[y,f(x,v)],u]+[x,[f(u,y),v]]+[x,[y,f(u,v)]].
$$
On the other hand, computing $f([x,u],[y,v])$ in a different way, it shows
$$
f([x,u],[y,v])=[[x,f(u,y)],v]+[[f(x,y),u],v]+[y,[f(x,v),u]]+[y,[x,f(u,v)]].
$$
By comparing the above two equations, we have
\begin{equation}\label{eight}
[f(x,y),[u,v]]+[[x,v],f(u,y)]=[[x,y],f(u,v)]+[f(x,v),[u,y]], \ \ \forall x,y,u,v\in L.
\end{equation}
Obviously, Eqs. (\ref{fourel}) and (\ref{eight})  are different. So the pivotal Lemma 2.2 in \cite{WD3} does not work. This tells us that
the problem about characterization of biderivation of a finite  dimensional complex simple Lie algebra is still open. In this paper, we work under the complex number field, and this field  also works with any algebraically closed field of characteristic zero.

It should be point out that Assumption \ref{zhangd2011} still holds if the biderivation $f$ is skewsymmetric (that is $f(x,y)=-f(y,x)$ for all $x,y\in L$). This result was given by Lemma 2.3 in \cite{Chen2016}.  However, the results of \cite{WD3} still hold if the biderivations are skewsymmetric. It is easy to check that Xu et. al. also forget the assumption of super-skewsymmetric   \cite{Xu2015}. Besides that, about this class of problems on biderivation of Lie (super-)algebra, $f$ is supposed to be (super-)skewsymmetric and applied the conclusion of Assumption \ref{zhangd2011} \cite{Chen2016,fan2016,Hanw,WD1,WD2}. In this paper, the assumption of skewsymmetric will not be applied.

In this paper, the notation concerning Lie algebras follows mainly from \cite{Car,Hum}.
Let $\mathbb{C}$ be the complex number field, $L$ a simple Lie algebra
over $\mathbb{C}$ of rank $l$, $\eta$ a fixed Cartan subalgebra of $L$, $\Phi \subseteq \eta^*$ the corresponding root
system of $L$, $\Delta$ a fixed base of $\Phi$, $\Phi_+$ (resp., $\Phi_-$) the set of positive (resp., negative)
roots relative to $\Delta$. The roots in $\Delta$ are called simple. Let
$L=\eta \oplus\left(\oplus_{\beta \in \Phi} L_\beta\right)$  be the root space decomposition of $L$, where
$L_{\beta}=\{x\in L| [h,x]= \beta(h)x, \ \forall h\in \eta \}$.
For each $\alpha\in \Phi$, take nonzero vectors $x_\alpha\in L_\alpha, h_\alpha\in \eta$
such that $x_{\alpha}, h_{\alpha}, x_{-\alpha}$
span a three-dimensional simple subalgebra of $L$ isomorphic to $sl_2(\mathbb{C})$, that is
$[x_\alpha, x_{-\alpha}]=h_{\alpha}$, $[h_\alpha,x_\alpha]=2x_\alpha$ and $[h_\alpha,x_{-\alpha}]=-2x_{-\alpha}$. The set
$\{h_{\alpha},x_\beta| \alpha \in \Delta, \beta\in \Phi \}$ forms a basis of $L$. For convenience, a new definition can be given as follows. That is, two roots $\alpha, \beta\in \Phi$ are called to be connected
if there are some $\gamma_1, \gamma_2,\cdots, \gamma_k\in \Phi$ satisfying $\alpha+\gamma_1, \gamma_1+\gamma_2, \cdots, \gamma_k+\beta\in \Phi\cup \{0\}$.
Obviously, the connected relation is an equivalence relation on $\Phi$. The following result is easy to prove and it will be useful in our main proof.

\begin{lemma}\label{connected}
Any two roots of a finite dimensional complex simple Lie algebra are connected.
\end{lemma}

\begin{proof}
Let $V$ be a real vector space of dimension $l$ with basis $\beta_1, \cdots, \beta_l$.
Let the symmetric scalar product $(,)$ be defined by $(\beta_i,\beta_j)=\delta_{i,j}$. The full root system $\Phi$ for each finite dimensional complex simple Lie algebra is well known (see \cite{Car}) and it can be listed as: 

$A_{l}:$ $\Phi=\{\beta_i-\beta_j| i,j=1,2, \cdots, l+1, i\neq j\}$, $(l\ge 1)$,

$B_{l}: $ $\Phi=\{\pm \beta_i\pm \beta_j, \pm \beta_i| i,j=1, \cdots, l,  i\neq j\}$, $(l\ge 2)$,

$C_{l}: $ $\Phi=\{\pm \beta_i\pm \beta_j, \pm 2\beta_i|  i,j=1, \cdots, l, i\neq j\}$, $(l\ge 3)$,

$D_{l}: $ $\Phi=\{\pm \beta_i\pm \beta_j |i,j=1, \cdots, l, i\neq j\}$, $(l\ge 4)$,

$E_6:$ $\Phi=\{\pm \beta_i\pm \beta_j | i,j=1,\cdots,5, i\neq j \}\cup
\{\frac{1}{2}\sum \limits_{i=1}^8 \epsilon_i\beta_i| \epsilon_i=\pm 1, \prod \limits_{i=1}^8 \epsilon_i=1, \epsilon_6=\epsilon_7=\epsilon_8\}$,

$E_7:$ $\Phi=\{\pm \beta_i\pm \beta_j |  i,j=1,\cdots,6, i\neq j\}\cup\{\pm (\beta_7+\beta_8)\}\cup
\{\frac{1}{2}\sum \limits_{i=1}^8 \epsilon_i\beta_i| \epsilon_i=\pm 1, \prod \limits_{i=1}^8 \epsilon_i=1, \epsilon_7=\epsilon_8\}$,

$E_8:$ $\Phi=\{\pm \beta_i\pm \beta_j | i,j=1,\cdots,8,i\neq j\}\cup
\{\frac{1}{2}\sum \limits_{i=1}^8 \epsilon_i\beta_i|  \epsilon_i=\pm 1,\prod \limits_{i=1}^8 \epsilon_i=1\}$,

$F_4:$ $\Phi=\{\pm \beta_i\pm \beta_j, \pm \beta_i| i,j=1,2,3,4, i\neq j\} \cup
\{\frac{1}{2}\sum \limits_{i=1}^4 \epsilon_i\beta_i| \epsilon_i=\pm 1\}$,

$G_2:$ $\Phi=\{\pm(-2\beta_i+\beta_j+\beta_k), \pm(\beta_i-\beta_j)| i,j,k=1,2,3 , i\neq j, j\neq k, i\neq k\}$.

Nextly, we can verify one by one that any two roots of one designated simple Lie algebra of types
$A_{l}$, $B_l$, $C_l$, $D_l$, $E_6$, $ E_7$, $E_8$, $F_4$, $G_2$ are connected.
\end{proof}

Now let us recall the definition of the derivation of a Lie algebra as follows.
\begin{definition}
Suppose that $L$ is a Lie algebra. A linear map $\phi: L\rightarrow L$ is called a derivation if it satisfies
\begin{eqnarray}
\phi([x,y])=[\phi(x),y]+[x,\phi(y)]
\end{eqnarray}
for all $x, y\in L$.
\end{definition}

For $x\in L$, it is easy to see that $\phi_x:L\rightarrow L, y\mapsto {\rm ad} x(y)=[x,y], $ for all $y\in L$ is a derivation of $L$.
It is called an inner derivation. The following result is well known.

\begin{lemma}\cite{Hum}\label{innerLie}
Every derivation of a finite dimensional complex simple Lie algebra is inner.
\end{lemma}

\section{Biderivations  of finite dimensional complex simple Lie algebras}

In this section, we assume that $f$ is a biderivation of a finite dimensional complex simple Lie algebra $L$. The following lemmas are useful to investigate the biderivation of $L$.
\begin{lemma}\label{phipsi}
There are two linear maps $\phi$ and $\psi$ from $L$ into itself such that
\begin{equation}\label{jianhua1}
f(x,y)=[\phi(x), y]=[x, \psi(y)], \ \ \forall x,y\in L.
\end{equation}
\end{lemma}

\begin{proof}
For the biderivation $f$ of $L$ and a fixed element $x\in L$, we define a map $\phi_x: L\rightarrow L$ given by $\phi_x(y)=f(x,y)$. Then, from Eq. (\ref{1der}) we see that $\phi_x$ is a derivation of $L$. Thus, by Lemma \ref{innerLie} we know that $\phi_x$ is an inner derivation of $L$. So there is a map
$\phi: L\rightarrow L$ such that $\phi_x={\rm ad}\phi(x)$, i.e.,  $f(x,y)=[\phi(x), y]$. Since $f$ is bilinear, thence $\phi$ can be proved to be linear. Similarly, we define a map $\psi_z$ from $L$ into itself given by $\psi_z(y)=f(y, z)$ for all $y\in L$. We can obtain a linear map $\psi$ from $L$ into itself such that $f(x,y)={\rm ad}(-\psi (y))(x)=[x, \psi(y)]$. The proof is completed.
\end{proof}

\begin{lemma}
For any $h\in \eta$, we have $\phi(h), \psi(h) \in \eta$, where $\phi$ and $\psi$ are defined by Lemma \ref{phipsi}.
\end{lemma}

\begin{proof}
For any $\alpha\in \Phi_+$, we choose $x_\alpha\in L_\alpha, x_{-\alpha}\in L_{-\alpha}$ and $h_\alpha\in \eta$ as above satisfying the Serre relations
$$
[x_\alpha, x_{-\alpha}]=h_\alpha, \ \ [h_\alpha, x_\alpha]=2x_\alpha, \ \ [h_\alpha, x_{-\alpha}]=-2x_{-\alpha}.
$$
For any but fixed $\alpha\in \Phi$,  here $\Phi\setminus \{\alpha, -\alpha\}$ is denoted by $\tilde{\Phi}_\alpha$. Let
\begin{eqnarray}
\phi(h_\alpha)&=&a_1h_1+a_2x_\alpha+a_3x_{-\alpha}+\sum_{\beta\in \tilde{\Phi}_\alpha} k_{\beta}x_\beta, \label{ee1}\\
\phi(x_\alpha)&=&b_1h_2+b_2x_\alpha+b_3x_{-\alpha}+\sum_{\beta\in \tilde{\Phi}_\alpha} t_{\beta}x_\beta, \label{ee2}\\
\phi(x_{-\alpha})&=&c_1h_3+c_2x_\alpha+c_3x_{-\alpha}+\sum_{\beta\in \tilde{\Phi}_\alpha} l_{\beta}x_\beta, \label{ee3} \\
\psi(h_\alpha)&=&s_1h_4+s_2x_\alpha+s_3x_{-\alpha}+\sum_{\beta\in \tilde{\Phi}_\alpha} m_{\beta}x_\beta, \label{ee4}\\
\psi(x_\alpha)&=&p_1h_5+p_2x_\alpha+p_3x_{-\alpha}+\sum_{\beta\in \tilde{\Phi}_\alpha} n_{\beta}x_\beta, \label{ee5}\\
\psi(x_{-\alpha})&=&q_1h_6+q_2x_\alpha+q_3x_{-\alpha}+\sum_{\beta\in \tilde{\Phi}_\alpha} r_{\beta}x_\beta \label{ee6}
\end{eqnarray}
for some $a_i,b_i,c_i,s_i,p_i, q_i, k_{\beta},t_{\beta},l_{\beta}, m_{\beta}, n_{\beta},r_{\beta}\in \mathbb{C}$, $i=1,2,3, \beta\in \tilde{\Phi}_\alpha$, and $h_j\in \eta$, $j=1,2, \cdots, 6$.
By Lemma \ref{phipsi} and Eqs. (\ref{ee1}), (\ref{ee5}), we have
\begin{equation}\label{aaa}
f(h_\alpha, x_\alpha)=[\phi(h_\alpha), x_\alpha]=a_1\alpha(h_1)x_{\alpha}-a_3h_{\alpha}+\sum_{\beta\in \tilde{\Phi}_\alpha} k_{\beta}[x_\beta, x_\alpha]
\end{equation}
and
\begin{equation}\label{-aaa}
f(h_\alpha, x_\alpha)=[h_\alpha, \psi(x_\alpha)]=2p_2x_{\alpha}-2p_3x_{-\alpha}+\sum_{\beta\in \tilde{\Phi}_\alpha} m_{\beta}\beta(h_\alpha)x_\beta.
\end{equation}
By comparing the above two equations, and $[x_\beta, x_\alpha]\in L_{\alpha+\beta}\neq \eta$ since $\beta\in  \tilde{\Phi}_\alpha$, we have
$a_3=0$. In the same way by considering $f(h_\alpha, x_{-\alpha})$ with Eqs. (\ref{ee1}) and (\ref{ee6}), $a_2=0$. Similarly, considering the images $f(x_\alpha, h_\alpha)$ and $f(x_{-\alpha}, h_\alpha)$, we have $s_2=s_3=0$ by Eqs. (\ref{ee2}), (\ref{ee3}) and (\ref{ee4}).
In other words, for $a_1h_1=\hat{h}^\alpha\in \eta$ and $s_1 h_4=\check{h}^\alpha\in \eta$, it can be seen that
\begin{eqnarray}
\phi(h_\alpha)&=&\hat{h}^\alpha+\sum_{\beta\in \tilde{\Phi}_\alpha} k_{\beta}x_\beta, \label{ee7}\\
\psi(h_\alpha)&=&\check{h}^\alpha+\sum_{\beta\in \tilde{\Phi}_\alpha} m_{\beta}x_\beta. \label{ee8}
\end{eqnarray}
Now for any root $\gamma\in \tilde{\Phi}_\alpha$, here $\Phi\setminus \{ \alpha, -\alpha, \gamma, -\gamma\}$ is denoted by $\tilde{\Phi}_{\alpha,\gamma}$.  By Eqs. (\ref{ee7}) and (\ref{ee8}), we can assume that
\begin{eqnarray}
\phi(h_\alpha)&=&\hat{h}^\alpha+k_\gamma x_\gamma+k_{-\gamma} x_{-\gamma}+\sum_{\beta\in \tilde{\Phi}_{\alpha,\gamma}} k_{\beta} x_\beta, \label{hint1}\\
\psi(h_\gamma)&=&\check{h}^\gamma+\mu_\alpha x_\alpha+\mu_{-\alpha} x_{-\alpha} +\sum_{\beta\in \tilde{\Phi}_{\alpha,\gamma}} \mu_{\beta} x_\beta.\nonumber
\end{eqnarray}
By Lemma \ref{phipsi} and the above two equations, it follows  that
\begin{eqnarray}
f(h_\alpha, h_\gamma)=[\phi(h_\alpha),h_\gamma]=-2k_\gamma x_\gamma+2k_{-\gamma} x_{-\gamma}-\sum_{\beta\in \tilde{\Phi}_{\alpha,\gamma}} k_{\beta}\beta(h_\gamma) x_\beta \label{ee9}
\end{eqnarray}
and
\begin{eqnarray}
f(h_\alpha, h_\gamma)=[h_\alpha, \psi(h_\gamma)]=2\mu_\alpha x_\alpha-2\mu_{-\alpha} x_{-\alpha} + \sum_{\beta\in \tilde{\Phi}_{\alpha,\gamma}} \mu_{\beta}\beta(h_\alpha) x_\beta. \label{ee10}
\end{eqnarray}
By comparing Eqs. (\ref{ee9}) and (\ref{ee10}), one has $k_\gamma=k_{-\gamma}=\mu_\alpha =\mu_{-\alpha}=0$.  Due to the arbitrariness of $\gamma$, we get
 $\phi(h_\alpha)=\hat{h}^\alpha\in \eta$ by Eq. (\ref{hint1}). Similarly, with the same idea, $\psi(h_\alpha)=\check{h}^\alpha\in \eta$. Since the vectors of set $\{h_\alpha, \alpha\in \Phi\}$ span $\eta$, we have $\phi(\eta)\subset \eta$ and $\psi(\eta)\subset \eta$, as desired.
\end{proof}

\begin{lemma}\label{KT}
Let $\phi$ and $\psi$ be defined by Lemma \ref{phipsi}. Then there is a complex number $\lambda$ such that
\begin{equation}
\phi(x)=\psi(x)=\lambda x, \forall x\in L_\alpha, \alpha\in \Phi.
\end{equation}
\end{lemma}
\begin{proof}
For every $\alpha\in \Phi$, it is easy to see that the set $\ker \alpha=\{h\in \eta| \alpha(h)=0\}$ is a proper subspace of $\eta$.
From our well-known result, we have  by $|\Phi|<+\infty$ that $\bigcup_{\alpha\in \Phi} \ker \alpha$ is a proper subset of $\eta$.
Thus, we can find a vector $h\in \eta\setminus \bigcup_{\alpha\in \Phi} \ker \alpha$. This indicates that $\beta (h)\neq 0$ for all $\beta \in \Phi$.
Choose a fixed root $\alpha$,  by Lemma \ref{phipsi} we deduce
\begin{equation}\label{88}
f(x_\alpha, h)=[\phi(x_\alpha), h]=-[h, \phi(x_\alpha)]
\end{equation}
and
\begin{equation}\label{99}
f(x_\alpha, h)=[x_\alpha, \psi(h)]=-[\psi(h), x_\alpha]=-\alpha(\psi(h))x_\alpha.
\end{equation}
Let
\begin{equation}\label{phix}
\phi(x_\alpha)=\sum_{\beta\in \Phi} t_\beta x_\beta + \hat{h}^\alpha,
\end{equation}
where $t_\beta\in \mathbb{C}$ and $\hat{h}^\alpha\in \eta$. This, together with Eq. (\ref{88}), gives that
\begin{equation}\label{100}
f(x_\alpha, h)=-\sum_{\beta\in \Phi} t_\beta \beta(h) x_\beta.
\end{equation}
With Eqs. (\ref{99}) and (\ref{100}), we have  $t_\beta \beta(h)=0$ for any
$\beta\in \Phi\setminus \{\alpha \}$ and $t_\alpha \alpha(h)=\alpha(\psi(h))$. With Eq. (\ref{phix}) and the fact $\beta(h)\neq 0,\alpha(h)\neq 0$, it implies that
\begin{equation}\label{101}
\phi(x_\alpha)=t_\alpha x_\alpha + \hat{h}^\alpha, \ \ t_\alpha=\frac{\alpha(\psi(h))}{\alpha(h)}.
\end{equation}
Similarly, by considering the image $f(h, x_\alpha)$, we have
\begin{equation}\label{102}
\psi(x_\alpha)=k_\alpha x_\alpha + \check{h}^\alpha, \ \ k_\alpha=\frac{\alpha(\phi(h))}{\alpha(h)},
\end{equation}
where $k_\alpha\in \mathbb{C}$ and $\check{h}^\alpha\in \eta$.
It should be pointed out that Eqs. (\ref{101}) and (\ref{102}) also tell us that
\begin{equation}\label{103}
t_{-\alpha}=t_\alpha, k_{-\alpha}=k_\alpha, \ \ \forall \alpha\in \Phi.
\end{equation}
For any $\alpha,\beta\in \Phi$, by Lemma \ref{phipsi} and Eqs. (\ref{101}), (\ref{102}), we have
\begin{equation}\label{104}
f(x_\alpha, x_\beta)=[\phi(x_\alpha), x_\beta]=t_\alpha[x_\alpha, x_\beta]+\beta(\hat{h}^\alpha)x_\beta
\end{equation}
and
\begin{equation}\label{105}
f(x_\alpha, x_\beta)=[x_\alpha, \psi(x_\beta)]=k_\beta[x_\alpha, x_\beta]-\alpha(\check{h}^\beta)x_\alpha.
\end{equation}
Combing Eq. (\ref{104}) with Eq. (\ref{105}), we firstly see that $\beta(\hat{h}^\alpha)=\alpha(\check{h}^\beta)=0$ if $\alpha\neq \beta$. But $-\alpha(\hat{h}^\alpha)=\alpha(\check{h}^{-\alpha})=0$ by taking $\beta=-\alpha$. This means that $\beta(\hat{h}^\alpha)=0$ for all $\beta\in \Phi$, i.e.,
$\hat{h}^\alpha\in \bigcap_{\beta\in \Phi} \ker \beta$,  which gives $\hat{h}^\alpha=0$. Similarly,  by taking $\alpha=-\beta$ we have $\check{h}^\beta=0$. Hence, by Eqs. (\ref{101}), (\ref{102}), one has that
\begin{equation}\label{106}
\phi(x_\alpha)=t_\alpha x_\alpha, \ \
\psi(x_\alpha)=k_\alpha x_\alpha.
\end{equation}
With $f(x_\alpha,x_{-\alpha})=[\phi(x_\alpha),x_{-\alpha}]=[x_\alpha,\psi(x_{-\alpha})]$ and Eq. (\ref{106}), it follows that
 \begin{equation}\label{1065}
 t_\alpha=k_{-\alpha}, \ \forall \alpha\in \Phi.
 \end{equation}
Then we review Eqs. (\ref{104}) and (\ref{105}), and get
 \begin{equation}\label{107}
t_\alpha=k_\beta, \ \ {\text for}\ \ \alpha+\beta\in \Phi.
\end{equation}
Eqs. (\ref{103}),(\ref{1065}) and (\ref{107}) imply that if $\alpha,\beta\in \Phi$ such that $\alpha+\beta\in \Phi\cup \{0\}$, then $t_\alpha=k_\beta=t_{\beta}=k_{\alpha}$. And,  for arbitrary connected roots $\alpha,\beta\in \Phi$, $t_\alpha=k_\beta=t_{\beta}=k_{\alpha}$. 
  By Lemma \ref{connected}, we conclude that
\begin{equation}\label{109}
t_\alpha=t_{\alpha^\prime}=k_\beta=k_{\beta^\prime}, \  \forall \alpha, \alpha^\prime, \beta, \beta^\prime \in \Phi.
\end{equation}
Let $t_\alpha=\lambda$ in Eq. (\ref{109}), the proof is completed.
\end{proof}

We now state our main result as follows.

\begin{theorem}\label{maintheo}
Suppose that $L$ is a finite dimensional  complex simple Lie algebra. Then $f$ is a biderivation of $L$ if  and only if it is inner, i.e., there is a complex number $\lambda$ such that
$$
f(x,y)=\lambda [x,y], \ \ \forall x,y\in L.
$$
\end{theorem}
\begin{proof}

The ``if'' direction is obvious. We now prove the ``only if'' direction.
By Lemma \ref{KT}, there is $\lambda\in \mathbb{C}$ such that
$
\phi(x)=\psi(x)=\lambda x, \forall x\in L_\alpha, \alpha\in \Phi,
$
where $\phi$ and $\psi$ are given by Lemma \ref{phipsi}. For any $h\in \eta$ and $\alpha\in \Phi$,
 by Lemma \ref{phipsi}, one shows that
$f(h,x_\alpha)=[\phi(h), x_\alpha]=[h, \psi(x_\alpha)]$. It implies that $\alpha(\phi(h))x_\alpha=\lambda \alpha(h)x_\alpha$. In other words, one has
$$
\alpha(\lambda h-\phi(h))=0, \ \ \forall \alpha \in \Phi,
$$
which yields $\phi(h)=\lambda h$ for all $h\in \eta$. Now, for any $x,y\in L$, let $x=\sum_{\alpha\in \Phi}l_\alpha x_\alpha+h$, where $l_\alpha\in \mathbb{C}$. Thus,
$$
f(x,y)=[\phi(x), y]=[\sum_{\alpha\in \Phi}l_\alpha \phi(x_\alpha)+\phi(h), y]=[\sum_{\alpha\in \Phi}l_\alpha \lambda x_\alpha+\lambda h, y]=[\lambda x, y].
$$
The proof is completed.
\end{proof}

\section{Applications}

\subsection{Biderivation of a general linear Lie algebra}

We firstly characterize the biderivation of a general linear Lie algebra $gl_n(\mathbb{C})$, which consisting all $n\times n$ complex matrices under the Lie product $[x,y]=xy-yx$ for all $x,y\in gl_n(\mathbb{C})$. Recall that $gl_n(\mathbb{C})$ has a Lie subalgebra
$sl_n(\mathbb{C})=\{x\in gl_n(\mathbb{C})| {\rm tr} (x)=0\}$, which is the $A_{n-1}$ type simple Lie algebra, where ${\rm tr} (x)$ denotes the trace of $x$. The following  fact is well known.

\begin{lemma}\label{glsl}
(i) $[gl_n(\mathbb{C}),gl_n(\mathbb{C})]=[sl_n(\mathbb{C}),sl_n(\mathbb{C})]=sl_n(\mathbb{C})$;

(ii)  $gl_n(\mathbb{C})=\mathbb{C}I_n\oplus sl_n(\mathbb{C})$. That is, for any $x\in gl_n(\mathbb{C})$, the following decomposition is unique:
 $$
 x= \frac{1}{n}{\rm tr} (x)I_n+ u_x
 $$
 where $u_x=x- \frac{1}{n}{\rm tr} (x)I_n \in sl_n(\mathbb{C})$ and
$I_n$ is the $n\times n$ identity matrix.
\end{lemma}

We have the following result by Theorem \ref{maintheo}.

\begin{theorem}\label{wo}
$f$ is a biderivation of $gl_n(\mathbb{C})$ if and only if there are two complex numbers $\lambda,\mu$ such that
$$
f(x,y)=\mu {\rm tr} (x) {\rm tr} (y)I_n+\lambda [x,y], \ \ \forall x,y\in gl_n(\mathbb{C}).
$$
\end{theorem}

\begin{proof}
The ``if'' direction is easy to verify. We now prove the ``only if'' direction.

We first claim that $\phi(I_n)\in \mathbb{C}I_n$ for any derivation $\phi$ of $gl_n(\mathbb{C})$. In fact, for all $x\in gl_n(\mathbb{C})$, by $[x,I_n]=0$ we have $0=\phi([x,I_n])=[\phi(x),I_n]+[x,\phi(I_n)]=[x,\phi(I_n)]$ . Hence, $\phi(I_n)\in Z(gl_n(\mathbb{C}))=\mathbb{C}I_n$. We have obtained the desired claim. Similar to the proof of Lemma \ref{phipsi}, if we let $\phi_x(y)=f(x,y)=\psi_y(x)$ for all $x,y\in gl_n(\mathbb{C})$, then $\phi_x$ and $\psi_y$ are both derivations of $gl_n(\mathbb{C})$. Thus, the above claim tells us that
\begin{equation}\label{incent}
f(x,I_n), f(I_n,y)\in \mathbb{C}I_n, \ \ \forall x,y\in gl_n(\mathbb{C}).
\end{equation}
On the other hand, according to Lemma \ref{glsl} (i), $x,y\in sl_n(\mathbb{C})$ can be written as $x=[x_1,x_2]$ and $y=[y_1,y_2]$, where $x_1,x_2,y_1,y_2\in sl_n(\mathbb{C})$. Thus, by the definition of biderivation, it is obtained that
$$
f(x,I_n)=f([x_1,x_2],I_n)=[f(x_1,I_n),x_2]+[x_1,f(x_2,I_n)]
$$
and
$$
f(I_n,y)=f(I_n, [y_1,y_2])=[f(I_n,y_1),y_2]+[y_1,f(I_n,y_1)].
$$
By (\ref{incent}), we conclude that $f(x_1,I_n), f(x_2,I_n),f(I_n,y_1)$ and $f(I_n,y_2)$ all lie in the center $\mathbb{C}I_n$. With the above equations, it yields that
\begin{equation}\label{I_nsl}
f(x,I_n)=f(I_n,y)=0, \ \ \forall x,y\in sl_n(\mathbb{C}).
\end{equation}
On the other hand, if $x,y\in sl_n(\mathbb{C})$,  we write as above $x=[x_1,x_2]$ for some $x_1,x_2\in sl_n(\mathbb{C})$.  Applying Eq. (\ref{1der}) one has that $f(x,y)=f([x_1,x_2],y)=[x_1,f(x_2,y)]+[f(x_1,y),x_2]\in sl_n(\mathbb{C})$. This means that $f$ restricts to $sl_n(\mathbb{C})\times sl_n(\mathbb{C})$ is also a biderivation of $sl_n(\mathbb{C})$.  By Theorem \ref{maintheo}, there is a complex number $\lambda$ such that
\begin{equation}\label{fxy=lm}
f(x,y)=\lambda [x,y], \ \ \forall x,y\in sl_n(\mathbb{C}).
\end{equation}
Now, for any $x,y\in gl_n(\mathbb{C})$, by Lemma \ref{glsl} (ii) we can write $x=\frac{1}{n}{\rm tr}(x)I_n+u_x$ and $y=\frac{1}{n}{\rm tr}(y)I_n+u_y$, where $u_x=x-\frac{1}{n}{\rm tr}(x)I_n, u_y=y-\frac{1}{n}{\rm tr}(y)I_n\in sl_n(\mathbb{C})$. Thus, 
\begin{eqnarray}
f(x,y)&=&f(\frac{1}{n}{\rm tr}(x)I_n+u_x,\frac{1}{n}{\rm tr}(y)I_n+u_y) \nonumber \\
&=&\frac{1}{n^2}{\rm tr}(x){\rm tr}(y)f(I_n,I_n)+\frac{1}{n}{\rm tr}(x)f(I_n,u_y)+\frac{1}{n}{\rm tr}(y)f(u_x,I_n)+f(u_x,u_y).\label{1000}
\end{eqnarray}
By Eq. (\ref{I_nsl}), we know that $f(I_n,u_y)=f(u_x,I_n)=0$.  Eq. (\ref{fxy=lm}) also tells us that $f(u_x,u_y)=\lambda[u_x,u_y]=\lambda[x-\frac{1}{n}{\rm tr}(x)I_n, y-\frac{1}{n}{\rm tr}(y)I_n]=\lambda[x,y]$. Hence, according to (\ref{incent}), if we let $f(I_n,I_n)=n^2\mu I_n$ for some $\mu \in \mathbb{C}$, then Eq. (\ref{1000}) becomes to
$$
f(x,y)=\mu {\rm tr}(x){\rm tr}(y)I_n+\lambda [x,y].
$$
The proof is completed.
\end{proof}

\begin{remark}
It is easy to see that the biderivation $f(x,y)=\mu {\rm tr} (x) {\rm tr} (y)I_n+\lambda [x,y]$ is non-inner. As far as we know, the examples of non-inner biderivation of a Lie algebra are limited, we can find in \cite{Hanw}. On the other hand, we want remind that this diderivation also is not skewsymmetric.
\end{remark}

\subsection{Linear commuting maps on Lie algebras}

Recall that a linear commuting map $\phi$ on a Lie algebra $L$ subject to $[\phi(x),x]=0$ for any $x\in L$. The first important result on linear (or additive ) commuting maps is Posner¡¯s theorem \cite{Pos} from 1957. Then many scholars study commuting maps on all kinds of algebra structures \cite{Bou,Bre2,Bre3,CWS,Chen2016,Fran,Hanw,WD1,WD2,XYi}. In particular, Bre\v{s}ar have been proposed briefly discuss various extensions of the notion of a commuting map \cite{Bre3}.

Obviously, if $\phi$ on
$L$ is such a map, then $[\phi(x), y] = [x, \phi(y)]$ for any $x, y\in L$. Define $f(x,y)=[\phi(x),y]=[x,\phi(y)]$, then it is easy to verify that $f$ is a biderivation of $L$. Similar to the proof of Theorem 3.1 in \cite{Chen2016}, we have the following result  by using Theorem \ref{maintheo} .

\begin{theorem}
Let $L$ be a finite dimensional complex simple Lie algebra. Then any linear map $\phi$ on $L$ is commuting if and only if it is a scalar
multiplication map on $L$.
\end{theorem}

And we also get   the following result by using Theorem \ref{wo}, which is a result of \cite{Fran}.

\begin{theorem}
Any linear map $\phi$ on $gl_n(\mathbb{C})$ is commuting if and only if there is a complex number $\lambda$ and a linear function $\sigma: gl_n(\mathbb{C})\rightarrow\mathbb{C}$ such that
$$
\phi(x)=\sigma(x) I_n+ \lambda x, \ \ \forall x\in gl_n(\mathbb{C}).
$$
\end{theorem}

\begin{proof}
The ``if'' direction is easy to verify. We now prove the ``only if'' direction.

By the above discussion we see that $f(x,y)=[\phi(x),y]$, $x,y\in gl_n(\mathbb{C})$ is a biderivation of $gl_n(\mathbb{C})$. By Theorem \ref{wo}, it shows that $[\phi(x),y]=\mu {\rm tr}(x){\rm tr}(y)I_n+\lambda [x,y]$ for some $\mu, \lambda\in \mathbb{C}$. Furthermore, we have $[\phi(x)-\lambda x,y]=\mu {\rm tr}(x){\rm tr}(y)I_n \in sl_n(\mathbb{C})\cap \mathbb{C}I_n=\{0\}$, i.e., $[\phi(x)-\lambda x,y]=0$ for all $x,y\in gl_n(\mathbb{C})$. This means that $\phi(x)-\lambda x\in Z(gl_n(\mathbb{C}))=\mathbb{C}I_n$. Thus, we can find a map $\sigma$ from $gl_n(\mathbb{C})$ into $\mathbb{C}$ such that
$$
\phi(x)-\lambda x=\sigma (x)I_n.
$$
It is easy to prove that $\sigma$ is linear. This completes the proof.
\end{proof}

\section*{ Acknowledgements}

I thank the referee and editor for their invaluable comments and suggestions. I also thank Dr. Kejia Zhang for the selfless assistance  in writing this article.

\section*{ Funding}

This work is supported in part by National Natural Science Foundation of China [grant number 11171294], Natural Science
Foundation of Heilongjiang Province of China [grant number A2015007], the fund of Heilongjiang Education Committee [grant number 12531483].

\end{document}